\documentclass[10pt,leqno]{amsart}
\theoremstyle{definition}
\usepackage{indentfirst, amssymb,amsmath,amsthm}
\usepackage{csquotes}
\usepackage{hyperref}
\usepackage{mathrsfs}
\newtheorem*{theoA}{Theorem A}
\newtheorem*{theoB}{Theorem B}
\newtheorem*{theoC}{Theorem C}
\newtheorem*{theoD}{Theorem D}
\newtheorem*{theoE}{Theorem E}
\newtheorem*{theoF}{Theorem F}
\newtheorem*{theoG}{Theorem G}
\newtheorem*{theoH}{Theorem H}
\newtheorem*{theoI}{Theorem I}
\newtheorem*{theoJ}{Theorem J}
\newtheorem{theo}{Theorem}[section]
\newtheorem{lem}{Lemma}[section]
\newtheorem{cor}{Corollary}[section]

\newtheorem{defi}{Definition}[section]
\newtheorem{rem}{Remark}[section]
\newtheorem{question}{Question}[section]
\newcommand{\ol}{\overline}
\newcommand{\be}{\begin{equation}}
\newcommand{\ee}{\end{equation}}
\newcommand{\beas}{\begin{eqnarray*}}
\newcommand{\eeas}{\end{eqnarray*}}
\newcommand{\bea}{\begin{eqnarray}}
\newcommand{\eea}{\end{eqnarray}}

\numberwithin{equation}{section}
\begin{document}
\title[value distribution of a Differential Monomial and some normality criteria]{On the value distribution of a Differential Monomial and some normality criteria}
\date{}
\author[W. L\"{u} and B. Chakraborty]{Weiran L\"{u} and Bikash Chakraborty}
\date{}
\address{Department of Mathematics, China University of Petroleum, Qingdao 266 580, P.R. China}
\email{luwr@upc.edu.cn}
\address{Department of Mathematics, Ramakrishna Mission Vivekananda Centenary College, Rahara,
West Bengal 700 118, India.}
\email{bikashchakraborty.math@yahoo.com, bikash@rkmvccrahara.org}
\maketitle
\let\thefootnote\relax
\begin{abstract} Let $f$ be a transcendental meromorphic function defined  in the complex plane $\mathbb{C}$, and $\varphi(\not\equiv 0,\infty)$ be a small function of $f$. In this paper, we give a quantitative estimation of the characteristic function $T(r, f)$ in terms of $N\left(r,\frac{1}{M[f]-\varphi(z)}\right)$ as well as $\ol{N}\left(r,\frac{1}{M[f]-\varphi(z)}\right)$, where $M[f]$ is the differential monomial, generated by $f$.\par
Moreover, we prove one normality criterion:  Let $\mathscr{F}$ be a family of analytic functions on a domain $D$ and let $k(\geq1)$,  $q_{0}(\geq 3)$, $q_{i}(\geq0)$ $(i=1,2,\ldots,k-1)$, $q_{k}(\geq1)$ be positive integers. If for each $f\in \mathscr{F}$, $f$ has only zeros of multiplicity at least $k$, and  $f^{q_{0}}(f')^{q_{1}}...(f^{(k)})^{q_{k}}\not=1$, then $\mathscr{F}$ is normal on domain $D$.
\end{abstract}
\section{Introduction}
In this paper, we use the standard notations of Nevanlinna theory (\cite{8}). Throughout this paper, we always assume that $f$ is a transcendental meromorphic function defined in the complex plane $\mathbb{C}$. It will be convenient to let that $E$ denote any set of positive real numbers of finite linear (Lebesgue) measure, not necessarily same at each occurrence. For any non-constant meromorphic function $f$, we denote by $S(r,f)$ any quantity satisfying $$S(r, f) = o(T(r, f))~~\text{as}~~r\to\infty,~r\not\in E.$$
\begin{defi}
Let $f$ be a non-constant meromorphic function. A meromorphic function $b(z)(\not\equiv 0,\infty)$ is called a \enquote{small function} with respect to $f$ if $T(r,b(z))=S(r,f)$.
\end{defi}
\begin{defi}(\cite{f})
Let $a\in \mathbb{C}\cup\{\infty\}$.  For a positive integer  $k$ and  for a complex constant $a$, We denote
\begin{enumerate}
\item [i)] by $N_{k)}\left(r,a;f\right)$ the counting function of $a$-points of $f$ with multiplicity $\leq k$,
\item [ii)] by $N_{(k}\left(r,a;f\right)$ the counting function of $a$-points of $f$ with multiplicity $\geq k$,
\item [iii)] by $N_{k}\left(r,a;f\right)$ the counting function of the $a$-points of $f$ with multiplicity $k$.
\end{enumerate}
Similarly, the reduced counting functions $\ol{N}_{k)}(r,a;f)$ and $\ol{N}_{(k}(r,a;f)$ are defined.
\end{defi}
In 1959, Hayman proved the following theorem:
\begin{theoA}(\cite{hn})
  If $f$ is a transcendental meromorphic function and $n\geq 3$, then $f^{n}f'$ assumes all finite values except possibly zero infinitely often.
\end{theoA}
Moreover, Hayman (\cite{hn}) conjectured that the Theorem A remains valid for the cases $n = 1,~ 2$. In 1979, Mues (\cite{m}) confirmed the Hayman's Conjecture for $n=2$ and Chen and Fang (\cite{chen}) ensured the conjecture for $n=1$ in 1995.\par
In 1992, Q. Zhang (\cite{qz}) gave the quantitative version of Mues's result  as follows:
\begin{theoB} For a transcendental meromorphic function $f$, the following inequality holds :
$$T(r,f)\leq 6N\bigg(r,\frac{1}{f^{2}f'-1}\bigg)+S(r,f).$$
\end{theoB}
In (\cite{xuyi}), Theorem B was improved by Xu and Yi as
\begin{theoC}(\cite{xuyi}) Let $f$ be a transcendental meromorphic function and $\phi(z)(\not\equiv 0)$ be asmall function, then
$$T(r,f)\leq 6\overline{N}\bigg(r,\frac{1}{\phi f^{2}f^{'}-1}\bigg)+S(r,f).$$
\end{theoC}
Also, Huang and Gu (\cite{hg}) extended Theorem B by replacing $f'$ by $f^{(k)}$, $k(\geq1)$ is an integer.
\begin{theoD}(\cite{hg}) Let $f$ be a transcendental meromorphic function and $k$ be a positive integer. Then
$$T(r,f)\leq 6N\bigg(r,\frac{1}{f^{2}f^{(k)}-1}\bigg)+S(r,f).$$
\end{theoD}
\begin{defi}(\cite{ld})
For a positive integer $k$, we denote $N^{\ast}_{k}(r,0;f)$ the counting function of zeros of $f$, where a zero of $f$ with multiplicity $q$ is counted $q$ times if $q\leq k$, and is counted $k$ times if  $q> k$.
\end{defi}
In 2003, I. Lahiri and S. Dewan (\cite{ld}) considered the value distribution of a differential polynomial in more general settings.
They proved the following theorem.
\begin{theoE}
Let $f$ be a transcendental meromorphic function and $\alpha=\alpha(z)(\not\equiv 0,\infty)$ be a small function of $f$. If  $\psi=\alpha(f)^{n}(f^{(k)})^{p}$, where $n(\geq 0)$ $p(\geq 1)$, $k(\geq 1)$ are integers, then for any small function $a=a(z)(\not\equiv 0,\infty)$ of $\psi$,
$$(p + n)T(r, f)\leq \overline{N}(r,\infty; f) + \overline{N}(r, 0; f) + pN^{\ast}_{k}(r, 0; f) + \overline{N}(r, a; \psi) + S(r, f).$$
\end{theoE}
The next theorem is an immediate consequence of the above theorem.
\begin{theoF} Let $f$ be a transcendental meromorphic function and $a$ be a non zero complex constant. Let $l\geq3$, $n\geq1$, $k\geq1$ be positive integers. Then
$$T(r,f)\leq \frac{1}{l-2}\ol{N}\bigg(r,\frac{1}{f^{l}(f^{(k)})^{n}-a}\bigg)+S(r,f).$$
\end{theoF}
In this direction, in 2009, Xu, Yi and Zhang (\cite{Xu}) proved the following theorem:
\begin{theoG}
 Let $f$ be a transcendental meromorphic function, and $k(\geq1)$ be a positive integer. If $N_{1}(r,0;f)=S(r,f)$, then
$$T(r,f)\leq 2\ol{N}\bigg(r,\frac{1}{f^{2}f^{(k)}-1}\bigg)+S(r,f).$$
\end{theoG}
Later, in 2011, Xu, Yi and Zhang (\cite{xu2}) removed the condition $N_{1}(r,0;f)=S(r,f)$ in above Theorem. They proved
\begin{theoH}
 Let $f$ be a transcendental meromorphic function, and $k(\geq1)$ be a positive integer.Then
$$T(r,f)\leq M\ol{N}\bigg(r,\frac{1}{f^{2}f^{(k)}-1}\bigg)+S(r,f),$$
where $M$ is $6$ if $k=1$, or $k\geq3$ and $M=10$ if $k=2$.
\end{theoH}
Recently, Karmakar and Sahoo(\cite{KS}) considered the value distribution of the differential polynomial $f^{n}f^{(k)}-1$ where $n(\geq 2)$ and $k(\geq 1)$ are integers. Also, Xu and Ye(\cite{xu3}) studied the value distribution of the differential polynomial $\varphi f^{2} (f')^{2}- 1$, where $f$ is a transcendental meromorphic function, and $\varphi(z)$ is a small function of $f(z)$.\par
Before going to furthermore, we need to introduce the following definition:
\begin{defi} Let $f$ be nonconstant meromorphic function defined in the complex plane $\mathbb{C}$. Also, let $q_{0},q_{1},...,q_{k}$ be $(k+1)$ $(k\geq1)$ non-negative integers and $a(z)$ be a small function of $f$. Then the expression defined by $$M[f]=a(z)(f)^{q_{0}}(f')^{q_{1}}...(f^{(k)})^{q_{k}}$$ is known as differential monomial generated by $f$.\par
In this context, the terms  $\mu:=q_{0}+q_{1}+...+q_{k}$ and $\mu_{*}:=q_{1}+2q_{2}+...+kq_{k}$ are known as the degree and weight of the differential monomial respectively.
\end{defi}
Since the differential monomial $M[f]$ is the general form of $(f)^{q_{0}}(f^{(k)})^{q_{k}}$, so from the above discussion, the following questions are natural:
\begin{question} Does there exist positive constants $B_{1}(>0)$ and $B_{2}(>0)$ such that
\begin{enumerate}
\item [i)] $T(r,f)\leq B_{1}~ N\bigg(r,\frac{1}{M[f]-1}\bigg)+S(r,f),$ ~~\text{and}
\item [ii)] $T(r,f)\leq B_{2}~\ol{N}\bigg(r,\frac{1}{M[f]-1}\bigg)+S(r,f)$ ~~\text{hold?}
\end{enumerate}
\end{question}
In this paper, we deal with these questions.
\section{Main Results}
\begin{theo}\label{th1.1} Let $f$ be a transcendental meromorphic function and $\varphi(z)(\not\equiv0,\infty)$ be a small function of $f$. If $q_{0}(\geq 1)$, $q_{i}(\geq0)$ $(i=1,2,\ldots,k-1)$, $q_{k}(\geq1)$ are integers. Then
\begin{eqnarray}
\nonumber &&\mu T(r,f)\\
\nonumber  &\leq& \ol{N}(r,\infty;f)+(\mu-q_{0})N(r,0;f)+(1+\mu_{\ast})\ol{N}(r,0;f)+\ol{N}(r,\varphi(z);M[f])+S(r,f).
\end{eqnarray}
\end{theo}
\begin{rem}
  Clearly Theorem \ref{th1.1} extends Theorem E.
\end{rem}
\begin{cor}\label{th1} Let $f$ be a transcendental meromorphic function and $\varphi(z)(\not\equiv0,\infty)$ be a small function of $f$. If $q_{0}(\geq 3+\mu_{\ast})$, $q_{i}(\geq0)$ $(i=1,2,\ldots,k-1)$, $q_{k}(\geq1)$ are integers. Then
\begin{eqnarray*}
T(r,f) &\leq& \frac{1}{q_{0}-2-\mu_{\ast}}\ol{N}\left(r,\frac{1}{M[f]-\varphi(z)}\right)+S(r,f).
\end{eqnarray*}
\end{cor}
\begin{rem}
  Clearly, Corollary  \ref{th1} extends Theorem F.
\end{rem}
\begin{cor}Let $f$ be a transcendental meromorphic function and $\varphi(z)(\not\equiv0,\infty)$ be a small function of $f$. If $q_{0}(\geq 3+\mu_{\ast})$, $q_{i}(\geq0)$ $(i=1,2,\ldots,k-1)$, $q_{k}(\geq1)$ are integers. Then $M[f]-\varphi(z)$ has infinitely many zeros.
\end{cor}
\begin{theo}\label{th2} Let $f$ be a transcendental meromorphic function and $\varphi(z)(\not\equiv0,\infty)$ be a small function of $f$. If $q_{0}(\geq 3)$, $q_{i}(\geq0)$ $(i=1,2,\ldots,k-1)$, $q_{k}(\geq1)$ are integers. Then
\beas  T(r,f) &\leq& \frac{1}{q_{0}-2}N\left(r,\frac{1}{M[f]-\varphi(z)}\right)+S(r,f).\eeas
\end{theo}
\begin{cor} Let $f$ be a transcendental meromorphic function and $\varphi(z)(\not\equiv0,\infty)$ be a small function of $f$. If $q_{0}(\geq 3)$, $q_{i}(\geq0)$ $(i=1,2,\ldots,k-1)$, $q_{k}(\geq1)$ are integers. Then $M[f]-\varphi(z)$ has infinitely many zeros.
\end{cor}
\begin{theo}\label{th3} Let $f$ be a transcendental meromorphic function and $\varphi(z)(\not\equiv0,\infty)$ be a small function of $f$. If every pole of $f$ has multiplicity atleast $l(\geq 1)$ and $q_{0}(\geq 1+\frac{2}{l})$, $q_{i}(\geq0)$ $(i=1,2,\ldots,k-1)$, $q_{k}(\geq1)$ are integers, then
\beas\label{133} T(r,f) &\leq& \frac{l}{lq_{0}-l-1}N\left(r,\frac{1}{M[f]-\varphi(z)}\right)+S(r,f).
\eeas
\end{theo}
\begin{cor}
Let $f$ be a transcendental meromorphic function having no simple pole and $\varphi(z)(\not\equiv0,\infty)$ be a small function of $f$. If $q_{0}(\geq2)$, $q_{i}(\geq0)$ $(i=1,2,\ldots,k-1)$, $q_{k}(\geq1)$ are integers, then $M[f]-\varphi(z)$ has infinitely many zeros.
\end{cor}
\section{Lemmas}
Let $M[f]=a(z)(f)^{q_{0}}(f')^{q_{1}}...(f^{(k)})^{q_{k}}$ be a differential monomial generated by a transcendental meromorphic function $f$ and $a(z)$ be a small function of $f$.\par
In this paper, we assume that $q_{0}(\geq1)$ and $q_{k}(\geq1)$ and $f$ is a transcendental meromorphic function.
\begin{lem}(\cite{yam})\label{yam}
Let $f$ be a non-constant meromorphic function on $\mathbb{C}$, and let $a_l,\ldots,a_q$ be distinct meromorphic functions on $\mathbb{C}$. Assume that $a_i$ are small functions with respect to $f$ for all $i=1,\ldots,q$. Then we have the second main theorem,
$$(q-2-\varepsilon)T(r,f)\leq\sum\limits_{i=1}^{q}\ol{N}(r,a_i,f)+S(r,f),$$
for all $\varepsilon>0$.
\end{lem}
\begin{lem}\label{lem1} For a non constant meromorphic function $g$, we obtain
$$N\left(r,\frac{g'}{g}\right)-N\left(r,\frac{g}{g'}\right)=\ol{N}(r,g)+N\left(r,\frac{1}{g}\right)-N\left(r,\frac{1}{g'}\right).$$
\end{lem}
\begin{proof} The proof is same as  the formula (12) of (\cite{jh}).
\end{proof}
\begin{lem}\label{lem1.5}(\cite{f})
Let $f$ be a transcendental  meromorphic function defined in the complex plane $\mathbb{C}$. Then
$$\lim\limits_{r\to\infty}\frac{T(r,f)}{\log r}=\infty.$$
\end{lem}
\begin{lem}\label{lem2} Let $M[f]$ be differential polynomial generated by a transcendental meromorphic function $f$. Then $M[f]$ is non-constant.
\end{lem}
\begin{proof}Here $$\left(\frac{1}{f}\right)^{\mu}=a(z)\left(\frac{f'}{f}\right)^{q_1}\left(\frac{f''}{f}\right)^{q_2}\ldots\left(\frac{f^{(k)}}{f}\right)^{q_{k}}\frac{1}{M[f]}.$$
Thus by the first fundamental theorem and lemma of logarithmic derivative, we have
\begin{eqnarray}
\nonumber \mu T(r,f) &\leq& \sum\limits_{i=1}^{k}q_{i}N\left(r,\left(\frac{f^{(i)}}{f}\right)\right)+T(r,M[f])+S(r,f)\\
\nonumber  &\leq& \sum\limits_{i=1}^{k}iq_{i}\left\{\ol{N}(r,0;f)+\ol{N}(r,\infty;f)\right\}+T(r,M[f])+S(r,f)\\
\nonumber  &\leq& \sum\limits_{i=1}^{k}iq_{i}\left\{N(r,0;M[f])+N(r,\infty;M[f])\right\}+T(r,M[f])+S(r,f)\\
\label{eqc1}  &\leq& (2\mu_{\ast}+1)T(r,M[f])+S(r,f),
\end{eqnarray}
Since $f$ is a transcendental meromorphic function, so by Lemma \ref{lem1.5} and inequlity (\ref{eqc1}),  $M[f]$ must be non-constant.
\end{proof}
\begin{lem}\label{lem3}
Let $f$ be a transcendental meromorphic function and $M[f]$ be a differential polynomial in $f$, then $$T\bigg(r,M[f]\bigg)=O(T(r,f))~~\text{and}~~S\bigg(r,M[f]\bigg)=S(r,f).$$
\end{lem}
\begin{proof} The proof is similar to the proof of the Lemma 2.4 of (\cite{ly}).
\end{proof}
\section {Proof of the Theorems}
\begin{proof} [\textbf{Proof of Theorem \ref{th1.1}}]
Now
$$\frac{1}{f^{\mu}}=\frac{M[f]}{f^{\mu}}\frac{1}{M[f]}.$$
Thus by the first fundamental theorem and lemma of logarithmic derivative, we have
\begin{eqnarray}
\label{111}  \mu T(r,f) &=& N(r,\frac{1}{f^{\mu}})+ m(r,\frac{1}{f^{\mu}})+O(1)\\
\nonumber    &\leq& N(r,0;f^{\mu})+m\left(r,\frac{1}{M[f]}\right)+S(r,f)\\
\nonumber  &\leq& N(r,0;f^{\mu})+T(r,M[f])-N(r,0;M[f])+S(r,f).
\end{eqnarray}
Now, by Lemma \ref{yam}, we have
\bea \label{112} &&T(r,M[f])\\
\nonumber &&\leq \ol{N}(r,0;M[f])+\ol{N}(r,\infty;M[f])+\ol{N}(r,\varphi(z);M[f])+\varepsilon T(r,f)+S(r,f),\eea
for all $\varepsilon>0$.\par
Let $z_{0}$ be a zero of $f$ with multiplicity $q(\geq1)$. We assume that $z_{0}$ is not a zero or pole of $\varphi(z)$. Now we consider two cases:\\
\textbf{Case-I} If $q\geq k+1$, then $z_{0}$ is a zero of $M[f]$ of order $q\mu-\mu_{\ast}$. Now
$$\mu q +1-(q\mu-\mu_{\ast})\leq(\mu-q_{0})q+(1+\mu_{\ast}).$$
\textbf{Case-II} If $q\leq k$, then $z_{0}$ is a zero of $M[f]$ of order $qq_{0}+(q-1)q_{1}+(q-2)q_{2}+\ldots+2q_{q-2}+q_{q-1}$. Now
\beas && \mu q +1-\{qq_{0}+(q-1)q_{1}+(q-2)q_{2}+\ldots+2q_{q-2}+q_{q-1}\}\\
&=&(\mu-q_{0})q+1+\{q_{1}+2q_{2}+\ldots+(q-2)q_{q-2}+(q-1)q_{q-1}\}-(qq_1+qq_2+\ldots+qq_{q-1})\\
&\leq&(\mu-q_{0})q+(1+\mu_{\ast}).
\eeas
Now, by the first fundamental theorem, we obtain
$$N(r,0;\varphi(z))\leq T(r,\varphi(z))+O(1)=S(r,f),$$
and,
$$N(r,\infty;\varphi(z))\leq T(r,\varphi(z))+O(1)=S(r,f),$$
so from the above dicussion, we have
\bea\label{113} && N(r,0;f^{\mu})+\ol{N}(r,0;M[f])-N(r,0;M[f])\\
\nonumber &\leq& (\mu-q_{0})N(r,0;f)+(1+\mu_{\ast})\ol{N}(r,0;f).\eea
Combining (\ref{111}),(\ref{112}) and (\ref{113}), we have
\begin{eqnarray}
&&\mu T(r,f)\\
\nonumber  &\leq& N(r,0;f^{\mu})+T(r,M[f])-N(r,0;M[f])+S(r,f)\\
\nonumber  &\leq& N(r,0;f^{\mu})+\ol{N}(r,0;M[f])+\ol{N}(r,\infty;M[f])+\ol{N}(r,\varphi(z);M[f])-N(r,0;M[f])+S(r,f)\\
\nonumber  &\leq& \ol{N}(r,\infty;f)+(\mu-q_{0})N(r,0;f)+(1+\mu_{\ast})\ol{N}(r,0;f)+\ol{N}(r,\varphi(z);M[f])+S(r,f).
\end{eqnarray}
This completes the proof.
\end{proof}
\begin{proof} [\textbf{Proof of Theorem \ref{th2}}]
Assume that $b=b(z)=\frac{1}{\varphi(z)}$. Now by Lemma \ref{lem2}, it is clear that $b(z)M[f]$ is non-constant. Again
 $$\frac{1}{f^{\mu}}=\frac{b M[f]}{f^{\mu}}-\frac{(b M[f])'}{f^{\mu}}\frac{(b M[f]-1)}{(b M[f])'}.$$
Thus in view of Lemmas \ref{lem1} and \ref{lem3}, the first fundamental theorem and lemma of logarithmic derivative, we have
\bea\label{kabir5} &&\mu m\left(r,\frac{1}{f}\right)\\
\nonumber &\leq& m\left(r,\frac{b M[f]}{f^{\mu}}\right)+m\left(r,\frac{(b M[f])'}{f^{\mu}}\right)+m\left(r,\frac{b M[f]-1}{(b M[f])'}\right)+O(1)\\
\nonumber&\leq& 2m\left(r,\frac{b M[f]}{f^{\mu}}\right)+m\left(r,\frac{(b M[f])'}{b M[f]}\right)+m\left(r,\frac{b M[f]-1}{(b M[f])'}\right)+O(1)\\
\nonumber&\leq& T\left(r,\frac{b M[f]-1}{(b M[f])'}\right)-N\left(r,\frac{b M[f]-1}{(b M[f])'}\right)+S(r,f)\\
\nonumber&\leq& \ol{N}(r,\infty;f)+N\left(r,\frac{1}{b M[f]-1}\right)-N\left(r,\frac{1}{(b M[f])'}\right)+S(r,f)\eea
Let us define $g:=b M[f]-1$ and $h:=\frac{g'}{f^{q_{0}-1}}$. Then
\bea\label{121} &&\mu T(r,f)\\
\nonumber &\leq& (\mu-q_{0}+1) N\left(r,\frac{1}{f}\right)+ \ol{N}(r,\infty;f)+N\left(r,\frac{1}{M[f]-\varphi(z)}\right)-N\left(r,0;h\right)+S(r,f).
\eea
If $q_{0}\geq3$, then from (\ref{121}), we have
\beas\label{1211}  T(r,f) &\leq& \frac{1}{q_{0}-2}N\left(r,\frac{1}{M[f]-\varphi(z)}\right)+S(r,f).\eeas
\end{proof}
\begin{proof} [\textbf{Proof of Theorem \ref{th3}}]
Using (\ref{121}) and the first fundamental theorem, we have
\bea\label{122} &&(\mu-q_{0}+1) m\left(r,\frac{1}{f}\right)+(q_{0}-1) T(r,f)\\
\nonumber &\leq& N\left(r,\frac{1}{M[f]-\varphi(z)}\right)+\ol{N}(r,\infty;f)-N\left(r,0;h\right)+S(r,f)\\
\nonumber &\leq& N\left(r,\frac{1}{M[f]-\varphi(z)}\right)+\frac{1}{l}N(r,\infty;f)+S(r,f).
\eea
That is,
\beas\label{123} &&T(r,f)\\
\nonumber &\leq& \frac{l}{lq_{0}-l-1}N\left(r,\frac{1}{M[f]-\varphi(z)}\right)+S(r,f).
\eeas
\end{proof}
This completes the proof.
\section{Applications}
Let $D\subset\mathbb{\overline{C}}$ be a domain. A family $\mathscr{F}$ of analytic functions in $D$ is said to be normal if every sequence $\{f_n\}\subset\mathscr{F}$ has a convergent subsequence, which converges spherically, locally and uniformly in $D$ to a analytic function or $\infty$.\par
The aim of this sectionis to provide normality criterion for a family of analytic functions. \par
Using  Mues's result (\cite{m}), Pang (\cite{pang}) proved the following result:
\begin{theoI}(\cite{pang})
Let $\mathscr{F}$ be a family of meromorphic function on a domain $D$. If each $f\in \mathscr{F}$  satisfies $f^2f'\not=1$, then $\mathscr{F}$ is normal on domain $D$.
\end{theoI}
In this sequel, in 2005, Huang and Gu (\cite{hg}) proved the following theorem:
\begin{theoJ}(\cite{hg})
Let $\mathscr{F}$ be a family of meromorphic functions on a domain $D$ and let $k$ be a positive integer. If for each $f\in \mathscr{F}$, $f$ has only zeros of multiplicity at least $k$ and $f^{2}f^{(k)}\not=1$, then $\mathscr{F}$ is normal on domain $D$.
\end{theoJ}
Using Theorem \ref{th2}, we provide a normality criterion for a family of analytic functions.
\begin{theo}\label{thn}
Let $\mathscr{F}$ be a family of analytic functions on a domain $D$ and let $k(\geq1)$,  $q_{0}(\geq 3)$, $q_{i}(\geq0)$ $(i=1,2,\ldots,k-1)$, $q_{k}(\geq1)$ be positive integers. If for each $f\in \mathscr{F}$,
\begin{enumerate}
  \item [i.] $f$ has only zeros of multiplicity at least $k$
  \item [ii.] $f^{q_{0}}(f')^{q_{1}}\ldots(f^{(k)})^{q_{k}}\not=1$,
\end{enumerate}
then $\mathscr{F}$ is normal on domain $D$.
\end{theo}
Before going to prove the above result, we need to recall a lemma.
\begin{lem}(\cite{sc})\label{holi}
  Let $\mathscr{F}$ be a family of meromorphic functions on the unit disc $\Delta$ such that all zeros of functions in $\mathscr{F}$ have multiplicity at least $k$. Let $\alpha$ be a real number satisfying $0\leq\alpha<k$. Then $\mathscr{F}$ is not normal in any neighbourhood of $z_0\in \Delta$ if and only if there exist
  \begin{enumerate}
    \item [(i)] points $z_k \in \Delta$, $z_k\rightarrow z_0$;
    \item [(ii)] positive numbers $\rho_k$, $\rho_k\rightarrow0$; and
    \item [(iii)] functions $f_k\in  \mathscr{F}$
  \end{enumerate}
such that $\rho_{k}^{-\alpha}f_{k}(z_k+\rho_{k}\zeta)\rightarrow g(\zeta)$ spherically uniformly on compact subsets of $\mathbb{C}$, where $g$ is a nonconstant meromorphic function.
\end{lem}
\begin{proof}[\textbf{Proof of the Theorem \ref{thn}}] Since normality is a local property, we may assume that $D=\Delta$. If possible, suppose that $\mathscr{F}$ is not normal on $\Delta$, then by Lemma \ref{holi}, there exist $\{f_n\}\subset\mathscr{F}$, $z_n\in \Delta$ and positive numbers $\rho_k$ with $\rho_k\rightarrow0$ such that
$$g_n(\zeta)=\rho_{n}^{-\alpha}f_{n}(z_n+\rho_{n}\zeta)\rightarrow g(\zeta)$$
locally, uniformly in spherical metric. Let
$$H_n(\zeta)=(g_{n}(\zeta))^{q_{0}}(g_{n}^{'}(\zeta))^{q_{1}}\ldots(g_{n}^{(k)}(\zeta))^{q_{k}}.$$
and,
$$H(\zeta)=(g(\zeta))^{q_{0}}(g^{'}(\zeta))^{q_{1}}\ldots(g^{(k)}(\zeta))^{q_{k}}.$$
Then \beas H_n(\zeta)&=&\rho_{n}^{\mu_{\ast}-\alpha\mu}(f_{n}(z_n+\rho_{n}\zeta))^{q_{0}}(f_{n}^{'}(z_n+\rho_{n}\zeta))^{q_{1}}\ldots(f_{n}^{(k)}(z_n+\rho_{n}\zeta))^{q_{k}}\\
&=& (f_{n}(z_n+\rho_{n}\zeta))^{q_{0}}(f_{n}^{'}(z_n+\rho_{n}\zeta))^{q_{1}}\ldots(f_{n}^{(k)}(z_n+\rho_{n}\zeta))^{q_{k}}~~\text{if~~we~~ choose}~~\mu_{\ast}-\alpha\mu=0\\
&\rightarrow& H(\zeta)~~\text{locally,~~ uniformly ~~in~~ spherical~~ metric.}
\eeas
Now, we  made the following observations:
\begin{enumerate}
  \item [a.] by Lemma \ref{holi}, $g(\zeta)$ is non-constant meromorphic function.
  \item [b.] by Hurwitz's Theorem (pp. 152, \cite{con}), all zeros of $g(\zeta)$ are of multiplicity atleast $k$.
  \item [c.] $H(\zeta)\not\equiv0$, otherwise, $g(\zeta)$ will become a polynomial of degree atmost $k-1$, which is impossible by (b).
  \item [d.] by Hurwitz's Theorem, $H(\zeta)\not=1$, as $H_{n}(\zeta)\not=1$.
  \item [e.] by Theorem \ref{th2} and (d), $g(\zeta)$ must be non-trancendental, i.e., non-constant rational function.
  \item [f.] Since $\mathscr{F}$ is a family of analytic functions, so $g_n(\zeta)$ is analytic. Since, $g_n(\zeta)\rightarrow g(\zeta)$
locally, uniformly in spherical metric and $g(\zeta)$ is non-constant, hence, $g(\zeta)$ is analytic.
\end{enumerate}
Thus using (e) and (f), we can conclude that $g(\zeta)$ is a non-constant polynomial function, say, $g(\zeta)=c_0+c_1\zeta+\ldots+c_l\zeta^{l}$. But by (b), $l$ must be atleast $k$. Thus $H(\zeta)$ is a non-constant polynomial. So, by the Fundamental Theorem of Algebra, $H(\zeta)=1$ has a solution, which contradicts (d). Thus our assumption is wrong. So $\mathscr{F}$ is normal. This completes the proof.
\end{proof}
Proceeding as above and using the result of Karmakar and Sahoo (Theorem 1.1 of \cite{KS}), the following result is obvious.
\begin{cor}
Let $\mathscr{F}$ be a family of analytic functions on a domain $D$ and let $k(\geq1)$ and $n(\geq 2)$ be two positive integers. If for each $f\in \mathscr{F}$,
\begin{enumerate}
  \item [i.] $f$ has only zeros of multiplicity at least $k$
  \item [ii.] $f^{n}f^{(k)}\not=1$,
\end{enumerate}
then $\mathscr{F}$ is normal on domain $D$.
\end{cor}
\section*{Acknowledgement}
We are thankful to Prof. C. C. Yang for giving us some ardent help and suggestion in the time of the preparation of this manuscript.

\end{document}